\documentclass[a4paper]{article}
\usepackage[utf8]{inputenc}

\usepackage{amsmath,amsthm,amssymb}
\usepackage{geometry}
\usepackage{color}
\usepackage{cite}
\usepackage{authblk}

\usepackage{tikz}
\usepackage{pgfplots}
\usepackage{subdepth}
\usepackage{enumerate}
\usepackage{authblk}
\usepackage{multirow}
\usepackage{mathdots}
\usepackage{mathtools}
\usepackage{graphicx}
\usepackage{siunitx}
\usepackage{hyperref}
\usepackage{xcolor}
\usepackage{subfig}

\newsavebox\solbox
\makeatletter
\newenvironment{sol}%
{
\@parboxrestore%
\begin{lrbox}{\solbox}%
\begin{minipage}{0.8\linewidth}
}
{\end{minipage}\end{lrbox}%
\begin{center}
\framebox{\usebox\solbox}
\end{center}
}
\makeatother

\usepackage{algorithm}
\usepackage{algpseudocode}

\usepackage{verbatim}
\usepackage{hyperref}
\usepackage{amsmath,amsthm,amssymb}
\usepackage{color}
\newtheorem{theorem}{Theorem}[section]
\newtheorem{lemma}[theorem]{Lemma}

\newcommand{\R}{\mathbb R}
\newcommand{\C}{\mathbb C}
\newcommand{\ii}{\mathfrak i}

\newcommand{\cT}{{\cal T}}
\newcommand{\diag}{{\rm diag}}

\newcommand{\orbr}{{\mathcal O}_{\R}}

\newcommand{\bun}{{\mathcal B}}
\newcommand{\bunr}{{\mathcal B}_{\R}}

\newcommand{\cbunr}{{\overline{\mathcal B}}_{\R}}

\newcommand{\cod}{{\rm codim}_\R}

\newcommand{\hide}[1]{}

\usepackage{tcolorbox}
\tcbuselibrary{skins}
\tcbuselibrary{breakable}
\tcbset{shield externalize}

\title{Generic real Jordan canonical forms}
\author[1]{Fernando De Ter\'an\thanks{\tt fteran@math.uc3m.es}}
\author[1] {Froil\'an M. Dopico\thanks{\tt dopico@math.uc3m.es}}
\affil[1]{Universidad Carlos III de Madrid, ROR: https://ror.org/03ths8210, Departamento de Matemáticas, Avda. Universidad 30, 28911, Legan\'es, Spain}

\date{}

\begin{document}

\maketitle

\begin{abstract}
    We obtain the generic real Jordan canonical forms for $n\times n$ matrices with real entries. More precisely, we prove that the set of $n\times n$ real matrices is the union of the closures of $\lfloor n/2\rfloor+1$ sets, which are called {\em generic bundles}, as they are particular ``bundles''.  In general, a bundle is the set of $n\times n$ real matrices with the same real Jordan canonical form, up to the values of the eigenvalues, provided that the eigenvalues which are distinct in one matrix of the bundle remain distinct in any other matrix of the same bundle. The $k$th generic bundle, for $0\leq k\leq\lfloor n/2\rfloor$, contains the $n\times n$ real matrices having $k$ different pairs of non-real conjugate eigenvalues and $n-2k$ different real eigenvalues. We prove that each of the $\lfloor n/2\rfloor+1$ generic bundles is an open subset of the set of $n\times n$ real matrices. Some numerical experiments are carried out with large sets of random matrices of different sizes to confirm that all the generic bundles show up, and only these ones.
\end{abstract}


\medskip

{\bf AMS subject classifications MSC2020.} 15A18, 15A20, 15A21, 15B52, 65F15.

\section{Introduction}\label{intro.sec}

It is well-known that most $n\times n$ matrices with complex entries have $n$ different eigenvalues. In the terminology of the present work, this can be rephrased as follows: the ``generic" Jordan canonical form for $n\times n$ complex matrices consists of a direct sum of $n$ Jordan blocks of size $1\times1$ corresponding to different eigenvalues. Recall that the Jordan form is a canonical form under {\em similarity} of complex matrices, where two matrices $A$ and $B$ are {\em similar} if there is some invertible matrix $P$ such that $B=P^{-1}AP$. It is natural to ask: what happens when the matrices have real entries and the allowed similarities are restricted to be real? Namely: which are the generic Jordan canonical forms of real $n\times n$ matrices  under real similarity? It is known that the eigenvalues of a real matrix are either real numbers or pairs of non-real conjugate numbers, but, how many of each kind appear generically in an $n\times n$ matrix? This is the question we aim to answer in this work from a topological point of view.

The {\em real Jordan canonical form} of $A$ (denoted by RJCF($A$)) is the representative of the real matrix $A$ under real similarity. It is a direct sum of blocks of two different kinds which correspond to, respectively, real eigenvalues and pairs of non-real conjugate eigenvalues of $A$ (see \cite[Th. 3.4.1.5]{hj13} and Section \ref{basic.sec} for the definitions). In this work, we determine the generic real Jordan canonical form of real $n\times n$ matrices. More precisely, we prove that there is a finite number of $n\times n$ RJCFs (which we explicitly describe) such that the set of matrices having these RJCFs is open and dense in $\R^{n\times n}$, that is, in the set of real $n\times n$ matrices. Each of these RJCFs is determined by the number of real eigenvalues.

Let us provide a more detailed description of the contribution of this work. The main result is Theorem \ref{main.th}, where we show that $\R^{n\times n}$ is the union of the closures of $\lfloor n/2\rfloor+1$ bundles under real similarity, which are different to each other (we prove a stronger fact, namely that the closure of a given generic bundle has empty intersection with any other generic bundle), and that these bundles are open in $\R^{n\times n}$ (in the standard topology). Therefore, the union of these bundles is an open dense set in $\R^{n\times n}$ and, so, it is generic in the standard topological sense. We call these $\lfloor n/2\rfloor+1$ bundles {\it generic} under real similarity. Each bundle corresponds to the set of real $n\times n$ matrices having the same number of different real eigenvalues, and the rest of eigenvalues being different couples of non-real conjugate eigenvalues. In terms of RJCFs, the $t$th bundle is a direct sum of $t$ real Jordan blocks of size $2\times 2$ associated with different couples of non-real conjugate eigenvalues together with $n-2t$ Jordan blocks of size $1\times 1$ associated with different real eigenvalues. Since $t$ ranges from $0$ to $\lfloor n/2\rfloor$, all possible numbers of real eigenvalues appear in the union. In other words, each generic bundle corresponds to an integer $k$, for $0 \leq k\leq n$, with the same parity as $n$, and contains all the $n\times n$ real matrices having exactly $k$ different real eigenvalues and $(n-k)/2$ different couples of non-real conjugate eigenvalues.

It is expected that the dimension of each generic bundle is $n^2$, which coincides with the dimension of $\R^{n\times n}$. By introducing an appropriate notion of dimension for the similarity bundles, and using the {\em codimension} within $\R^{n\times n}$ instead, we prove that this is indeed the case, namely we show that the codimension of each generic bundle is $0$.

The problem we address in this paper is related to the one of determining which is the expected RJCF of random matrices. Mimicking the previous considerations for the generic RJCF, it is natural to expect that a random real matrix has all its eigenvalues different to each other, so the remaining question is: how many of them are real? This question has been addressed in several previous works. In \cite{edelman97}, Edelman obtained, for every $0 \leq k\leq n$, the probability of a random $n\times n$ real matrix to have exactly $k$ real eigenvalues, whereas in \cite{eks94} the expected number of real eigenvalues was calculated (random matrices in these works are matrices whose entries follow independent standard normal distributions, though it is mentioned that, after extensive numerical experience, similar results can be obtained for other distributions). According to our results, all values $0 \leq k\leq n$ having the same parity as $n$ must have a positive probability, and they should show up experimentally in a sufficiently large number of tests. We confirm this fact with several numerical experiments.

\section{Notation and basic definitions}\label{basic.sec}

Following the notation in \cite{hj13}, by $C(a,b):=\begin{bsmallmatrix}
        \phantom{-}a&\phantom{!}b\\-b&\phantom{!}a
    \end{bsmallmatrix}$ we denote a $2\times2$ real Jordan block associated with a couple of non-real conjugate eigenvalues $a\pm b\ii$, with $a,b\in\R$ and $b>0$ (from now on, $\ii$ denotes the imaginary unit). Also, we use the following notation for, respectively, a $k\times k$ Jordan block associated with the eigenvalue $\mu\in\C$ and a $2k\times 2k$  real Jordan block associated with a couple of non-real conjugate eigenvalues $a\pm b\ii$, with $a,b\in\R$ and $b>0$:
\begin{equation}\label{jordanblocks}
    J_k(\mu)=\begin{bmatrix}
        \mu&1\\&\ddots&\ddots\\&&\mu&1\\&&&\mu
    \end{bmatrix}_{k\times k},\qquad C_k(a,b)=\begin{bmatrix}
        C(a,b)&I_2\\&\ddots&\ddots\\&&C(a,b)&I_2\\&&&C(a,b)
    \end{bmatrix}_{2k\times2k}.
\end{equation}
When $k=1$, the block $J_1(\mu)$ will be written as $[\mu]$ for simplicity, as it is just a number.

For $A\in\R^{n\times n}$, its {\em real orbit under similarity} is the set of real matrices which are similar to $A$, namely,
\begin{equation}\label{orbit}
    \orbr(A)=\{P^{-1}AP:\ \ \mbox{$P\in\R^{n\times n}$ invertible}\}.
\end{equation}

The orbit $\orbr(A)$ is a differentiable manifold over $\R$, and its tangent space at $A$ is the set (see, for instance, \cite[\S 4.1]{de95}):
$$
{\cal T}_A=\{XA-AX:\ X\in\R^{n\times n}\}.
$$
Therefore, the dimension of $\cT_A$, denoted by $\dim_\R(\cT_A)$, is the dimension of the  real vector space of matrices of the form $XA-AX$. This is the dimension of the orbit $\orbr(A)$. We will consider instead the {\em codimension} of $\orbr(A)$, namely the dimension of the normal space to $\cT_A$, which is equal to $\cod\cT_A=n^2-\dim_\R\cT_A$, as $n^2$ is the dimension of the ambient space $\R^{n\times n}$. Moreover,
\begin{equation} \label{eq.codimorb}
\cod\orbr(A)=\dim_\R\{X\in\R^{n\times n} \, : \, XA-AX=0  \},
\end{equation}
namely the codimension of the orbit $\orbr(A)$ is the dimension of the solution space of the linear equation $XA-AX=0$, which is a real vector space (see, for instance, \cite[p. 71]{de95}).

The {\em real bundle under similarity} of $A\in\R^{n\times n}$, denoted by $\bunr(A)$, is the set of real matrices which have the same RJCF as $A$, up to the specific values of the eigenvalues, provided that the eigenvalues which are distinct in one matrix remain distinct in the other ones. More precisely, if
\begin{equation}\label{rjcf}
    {\rm RJCF}(A)=\bigoplus_{i=1}^r\left(\bigoplus_{j=1}^{d_i}C_{\ell_{j,i}}(\tilde a_i, \tilde b_i)\right)\oplus\bigoplus_{i=1}^s\left(\bigoplus_{j=1}^{e_i} J_{k_{j,i}}(\tilde c_i)\right),
\end{equation}
with $\tilde a_i,\tilde b_i >0, \tilde c_i \in \mathbb{R}$, and $\tilde c_i\neq \tilde c_{i'},\  (\tilde a_i, \tilde b_i) \neq (\tilde a_{i'}, \tilde b_{i'})$, for $i\neq i'$, then the real bundle of $A$ is
\begin{equation} \label{bundle}
    \bunr(A):= \bigcup_{\substack{a_i, \, b_i>0, \, c_i \, \in \, \R , \\ c_i\neq c_{i'}, \, \, \mathrm{if} \, i\ne i', \\ (a_i,b_i) \neq (a_{i'}, b_{i'}) , \, \, \mathrm{if} \, i\ne i'} }  \orbr
    \left(\bigoplus_{i=1}^r\left(\bigoplus_{j=1}^{d_i}C_{\ell_{j,i}}(a_i,b_i)\right)\oplus\bigoplus_{i=1}^s\left(\bigoplus_{j=1}^{e_i} J_{k_{j,i}}(c_i)\right)\right).
\end{equation}

Following Arnold, \cite[\S5.5]{arnold71}, the codimension of the bundle $\bunr(A)$ is defined as the difference between the codimension of the orbit and the number of different eigenvalues in any matrix of the bundle. More precisely, if  RJCF($A$) is as in \eqref{rjcf}, then
\begin{equation} \label{eq.codimbund}
\cod\bunr(A)=\cod\orbr(A)-(2r+s).
\end{equation}

Since bundles under similarity are subsets of $\R^{n\times n}$, the topological notions (like the {\em closure} or the openness) of bundles are considered in the standard (Euclidean) topology of $\R^{n\times n}$. The closure of $\bunr(A)$ will be denoted by $\cbunr(A)$.

\section{The main result}\label{main.sec}

In Theorem \ref{main.th} we provide the generic RJCFs of real $n\times n$ matrices. This is done by showing that $\R^{n\times n}$ is the union of the closures of $\lfloor n/2\rfloor+1$ bundles, which correspond to some particular RJCFs. These are the ``generic" bundles (in each of these RJCFs the eigenvalues are different to each other and the number of real eigenvalues determines each RJCF). Moreover, we see that the intersection of the closure of any of the generic bundles with any other generic bundle is empty, which implies that none of these closures is included in any other. We also prove that the generic bundles are open sets. Note that this, in particular, implies that the union of these bundles is an open and dense set in $\R^{n\times n}$. Finally, we show that the codimension of each of the generic bundles is equal to  zero.

To prove Theorem \ref{main.th} we will use the following technical lemma.

\begin{lemma}\label{nrealevals.lemma}
 If $M\in\cbunr \left(\bigoplus_{i=1}^tC(a_i,b_i)\oplus\bigoplus_{i=1}^{n-2t}[c_i]\right)$, for some $0\leq t\leq\lfloor n/2\rfloor$, with $a_i, b_i, c_i\in\R$, $ b_i >0$, then
 \begin{itemize}
     \item[\rm (a)]$M$ has, at least, $n-2t$ real eigenvalues, counting multiplicities, and
     \item[\rm(b)] if $M$ has more than $n-2t$ real eigenvalues, counting multiplicities, then at least one of the eigenvalues of $M$ is multiple.
 \end{itemize}
\end{lemma}
\begin{proof}
    If $M\in\cbunr \left(\bigoplus_{i=1}^tC(a_i,b_i)\oplus\bigoplus_{i=1}^{n-2t}[c_i]\right)$, there is a sequence $\{M_m\}_{m\in\mathbb N}$ which converges to $M$ and such that $M_m\in\bunr\left(\bigoplus_{i=1}^tC(a_i,b_i)\oplus\bigoplus_{i=1}^{n-2t}[c_i]\right)$, for all $m\in\mathbb N$. Let 
    ${\rm RJCF}(M_m)=\bigoplus_{i=1}^tC(a_{i,m},b_{i,m})\oplus\bigoplus_{i=1}^{n-2t}[c_{i,m}]$, for some $a_{i,m},b_{i,m},c_{i,m}\in\R$,  $b_{i,m} >0$. By the real Schur factorization (see, for instance, \cite[Th. 2.3.4-(b)]{hj13}), there is an orthogonal matrix $Q_m \in \R^{n\times n}$ such that
    \begin{equation}\label{mk}
Q_m^\top M_mQ_m=\begin{bmatrix}
T_{1,m}&*&*&*&\cdots&*\\&\ddots&*&*&\cdots&*\\&& T_{t,m} \\&&&c_{1,m}&\ddots&\vdots\\&&&&\ddots&*\\0&&&&&c_{n-2t,m}
\end{bmatrix} \in \R^{n \times n},
\end{equation}
where $T_{i,m} \in \R^{2 \times 2}$ is real similar to $C(a_{i,m},b_{i,m})$ for $1 \leq i \leq t$, the block lower triangular part of the right-hand side matrix in \eqref{mk} is zero and the entries of the block upper triangular part, marked with $*$, are not of interest in our developments. Since the orthogonal group is a compact set, taking a subsequence if necessary, the sequence $\{Q_m\}_{m\in\mathbb N}$ converges to some orthogonal matrix $Q$. Therefore, $\{Q_m^\top M_mQ_m\}_{m\in\mathbb N}$ converges to $Q^\top MQ$, which is similar to $M$ and has the same block upper triangular structure as \eqref{mk}. If the block diagonal part of $Q^\top MQ$ is $\diag (T_1, \ldots, T_t, c_1, \ldots c_{n-2t})$, where $T_i \in \R^{2 \times 2}$ and $c_i \in \R$, then $T_i =\lim_{m \rightarrow \infty} T_{i,m}$ and $c_i =\lim_{m \rightarrow \infty} c_{i,m}$. Moreover, if $\alpha_i , \beta_i$ are the two eigenvalues of $T_i$, then $\alpha_i = \lim_{m \rightarrow \infty} (a_{i,m} + b_{i,m} \ii)$ and  $\beta_i  = \lim_{m \rightarrow \infty} (a_{i,m} - b_{i,m} \ii)$, for $1 \leq i \leq t$, by the continuity of the eigenvalues \cite[Th. 2.4.9.2]{hj13}. Since the eigenvalues of $M$ are  $\alpha_1 , \beta_1, \ldots, \alpha_t , \beta_t, c_1, \ldots c_{n-2t}$ and $c_i\in\R$, because $c_{i,m}\in\R$, we get that the matrix $M$ has, at least, $n-2t$ real eigenvalues. This proves (a).

To prove (b), assume that $M$ has more than $n-2t$ real eigenvalues.
For this to happen, it must be $\lim_{m \rightarrow \infty} b_{i,m} = 0$, for some $1\leq i\leq t$. But then $\alpha_i = \beta_i = \lim_{m \rightarrow \infty} a_{i,m}$ is a multiple (at least double) eigenvalue of $M$.
\end{proof}

Now we are in the position to state and prove Theorem \ref{main.th}. The first term of the direct sum in the right-hand side of \eqref{mainid} is empty when $t=0$, and the same happens with the second term when $n-2t=0$, namely when $t=n/2$.

\begin{theorem}\label{main.th}
    The set of real $n\times n$ matrices is equal to the following finite union of bundle closures under similarity:
    \begin{equation}\label{mainid}
    \R^{n\times n}=\bigcup_{t=0}^{\lfloor n/2\rfloor}\cbunr \left(\bigoplus_{i=1}^tC(a_i,b_i)\oplus\bigoplus_{i=1}^{n-2t}[c_i]\right),
    \end{equation}
    with $a_i,b_i,c_i\in\R$, $b_i >0$, and where $(a_i,b_i) \neq (a_{i'}, b_{i'})$ and $c_i\neq c_{i'}$ for $i\neq i'$.

    Moreover:
    \begin{itemize}
        \item[\rm(i)] $\cod\bunr \left(\bigoplus_{i=1}^tC(a_i,b_i)\oplus\bigoplus_{i=1}^{n-2t}[c_i]\right)=0$, for all $0\leq t\leq \lfloor n/2\rfloor$,
        \item[\rm(ii)]  $\bunr \left(\bigoplus_{i=1}^tC(a_i,b_i)\oplus\bigoplus_{i=1}^{n-2t}[c_i]\right)\cap\cbunr \left(\bigoplus_{i=1}^{t'}C(a_i,b_i)\oplus\bigoplus_{i=1}^{n-2t'}[c_i]\right)=\emptyset$, for $t\neq t'$, and
        \item[\rm(iii)] $\bunr \left(\bigoplus_{i=1}^tC(a_i,b_i)\oplus\bigoplus_{i=1}^{n-2t}[c_i]\right)$ is open, for all $0\leq t\leq\lfloor n/2\rfloor$.
    \end{itemize}
\end{theorem}
\begin{proof}
    Let us first prove the identity \eqref{mainid}. For this, let $A\in\R^{n\times n}$. Assume that ${\rm RJCF}(A)$ is as in \eqref{rjcf}, so there is some invertible matrix $P\in\R^{n\times n}$ such that $P^{-1}AP={\rm RJCF}(A)$. We are going to see that $A\in \cbunr \left(\bigoplus_{i=1}^{n_1}C(a_i,b_i)\oplus \bigoplus_{i=1}^{n_2}[c_i]\right)$, where $n_1=\sum_{i=1}^r\sum_{j=1}^{d_i}\ell_{j,i}$ and $n_2=\sum_{i=1}^s\sum_{j=1}^{e_i}k_{j,i}$. For this, we consider the following sequence $\{ A_m \}_{m \in \mathbb{N}}$, constructed as a perturbation of $A = P\cdot{\rm RJCF}(A)\cdot P^{-1}$:
    $$
    \begin{array}{ccl}
    A_m&=&P\left(\displaystyle\displaystyle\bigoplus_{i=1}^r\left(\bigoplus_{j=1}^{d_i}C_{\ell_{j,i}}(\tilde a_i, \tilde b_i)+\begin{bmatrix}
        C(\frac{1}{m+j},\frac{1}{m+j})\\&C(\frac{1}{2m+j},\frac{1}{2m+j})\\&&\ddots\\&&&C(\frac{1}{\ell_{j,i}m+j},\frac{1}{\ell_{j,i}m+j})
    \end{bmatrix}\right)\right.\\&&\left.\displaystyle\oplus \bigoplus_{i=1}^s\left(\bigoplus_{j=1}^{e_i} J_{k_{j,i}} (\tilde c_i) +\begin{bmatrix}
        \frac{1}{m+j}\\&\frac{1}{2m+j}\\&&\ddots\\&&&\frac{1}{k_{j,i}m+j}
    \end{bmatrix}\right)\right)P^{-1}.
    \end{array}
    $$
Note that the sequence $\{A_m\}_{m\in\mathbb N}$ converges to $A = P\cdot {\rm RJCF}(A)\cdot P^{-1}$. Moreover, $A_m\in \bunr \left(\bigoplus_{i=1}^{n_1}C(a_i,b_i)\oplus \bigoplus_{i=1}^{n_2}[c_i]\right)$ for $m$ large enough, with $c_i\neq c_{i'}$ and $(a_i , b_i) \neq (a_{i'}, b_{i'})$, for $i\neq i'$, because of the following:
\begin{itemize}
    \item The real eigenvalues of $A_m$ are $ \tilde c_i+\frac{1}{km+j}$, where $1 \leq i \leq s$, $1\leq j\leq e_i$, and $1\leq k\leq k_{j,i}$. Let us see that all of them are distinct for $m$ sufficiently large. If $\tilde c_i+\frac{1}{km+j}=  \tilde c_i+\frac{1}{k'm+j'}$, for some $1\leq k \leq k_{j,i}\,$,  $1\leq k'\leq k_{j',i}\,$, and $1\leq j,j'\leq e_i\,$, then it must be $(k-k')m=j'-j$. But, for $m$ large enough, this is not possible unless $k=k'$ and $j=j'$, because  $|j'-j|\leq e_i$. For $i\neq i'$, since $\tilde c_i\neq \tilde c_{i'}$, a value of $m$ large enough guarantees that $\tilde c_i+\frac{1}{km+j}\neq \tilde c_{i'}+\frac{1}{k'm+j'}$, for any $1\leq k\leq k_{j,i},1\leq j\leq e_i,1\leq k'\leq k_{j',i'}$, and $1\leq j'\leq e_{i'}$, because $\frac{1}{km+j}$ and $\frac{1}{k'm+j'}$ tend to $0$ as $m$ tends to infinity.
    \item Similarly, the non-real complex conjugate eigenvalues of $A_m$ are $\tilde a_i+\frac{1}{\ell m+j} \pm (\tilde b_i+\frac{1}{\ell m+j})\ii$, for $1\leq i\leq r$, $1\leq j\leq d_i$, and $1\leq \ell\leq\ell_{j,i}$. To see that all of them are distinct for $m$ large enough, it is sufficient to check that those corresponding to the ``$+$'' sign are all distinct, since $\widetilde b_i >0$. If $\tilde a_i+\frac{1}{\ell m+j}+ (\tilde b_i+\frac{1}{\ell m+j})\ii= \tilde a_{i}+\frac{1}{\ell'm+j'}+ (\tilde b_{i}+\frac{1}{\ell'm+j'})\ii$, for some $1\leq \ell \leq \ell_{j,i}\,$,  $1\leq \ell' \leq \ell_{j',i}\,$, and $1\leq j,j'\leq d_i\,$, we conclude, as before, that, for $m$ large enough, it must be $\ell=\ell'$ and $j=j'$. Also, for $i \ne i'$, since $(\tilde a_i, \tilde b_i )\neq (\tilde a_{i'}, \tilde b_{i'})$, we get that, for $m$ large enough, $\tilde a_i+\frac{1}{\ell m+j}+ (\tilde b_i+\frac{1}{\ell m+j})\ii\neq \tilde a_{i'}+\frac{1}{\ell'm+j'}+ (\tilde b_{i'}+\frac{1}{\ell'm+j'})\ii$, for all $1\leq \ell\leq\ell_{j,i},1\leq\ell'\leq \ell_{j',i'}\,,$ and $1\leq j\leq d_i,1\leq j'\leq d_{i'}$. Finally, note that $\tilde b_i+\frac{1}{\ell m+j} > 0$ for $m$ sufficiently large because $\tilde b_i > 0$.
\end{itemize}
Therefore, $A\in \cbunr \left(\bigoplus_{i=1}^{n_1}C(a_i,b_i)\oplus \bigoplus_{i=1}^{n_2}[c_i]\right)$, with $n_1$ and $n_2$ as above and $a_i,b_i,c_i\in\R$, $b_i >0$, $(a_i,b_i) \neq (a_{i'}, b_{i'})$, and $c_i\neq c_{i'}$, for $i\neq i'$. Since it must be $2n_1+n_2=n$, if we set $t=n_1$, then $A\in \cbunr \left(\bigoplus_{i=1}^tC(a_i,b_i)\oplus\bigoplus_{i=1}^{n-2t}[c_i]\right)$, and this proves \eqref{mainid}.

Now let us prove claim (ii) in the statement. Assume that there is some matrix $A\in\bunr \left(\bigoplus_{i=1}^tC(a_i,b_i)\oplus\bigoplus_{i=1}^{n-2t}[c_i]\right)\cap\cbunr \left(\bigoplus_{i=1}^{t'}C(a_i,b_i)\oplus\bigoplus_{i=1}^{n-2t'}[c_i]\right)$, for some $1\leq t,t'\leq\lfloor n/2\rfloor$.

Since $A\in\bunr \left(\bigoplus_{i=1}^tC(a_i,b_i)\oplus\bigoplus_{i=1}^{n-2t}[c_i]\right)$, then $A$ has exactly $n-2t$ real eigenvalues. Moreover, since $A\in\cbunr \left(\bigoplus_{i=1}^{t'}C(a_i,b_i)\oplus\bigoplus_{i=1}^{n-2t'}[c_i]\right)$, 
by part (a) of Lemma \ref{nrealevals.lemma}, $A$ has, at least, $n-2t'$ real eigenvalues, so it must be $n-2t\geq n-2t'$, namely $t'\geq t$. If $t'>t$, then part (b) in Lemma \ref{nrealevals.lemma} implies that $A$ has at least one multiple eigenvalue, which is not the case. Therefore, it must be $t'=t$.

 Let us now prove claim (iii). For simplicity, set $\bun_t:=\bunr \left(\bigoplus_{i=1}^tC(a_i,b_i)\oplus\bigoplus_{i=1}^{n-2t}[c_i]\right)$, with $a_i, b_i$, and $c_i$ as in the statement, for $0\leq t\leq\lfloor n/2\rfloor$. Let $M\in\bun_{t_0}$, for some $0\leq t_0\leq\lfloor n/2\rfloor$. We are going to see that there is some $\varepsilon>0$ such that $B(M,\varepsilon)\subseteq \bun_{t_0}$, where $B(M,\varepsilon)=\{A\in\R^{n\times n}:\ \|M-A\|_2<\varepsilon\}$ is the (open) ball of radius $\varepsilon$ centered at $M$ (and where $\|\cdot\|_2$ denotes the spectral norm, see, for instance, \cite[Example 5.6.6]{hj13}). First, by the continuity of the eigenvalues (see, for instance, \cite[Th. D.2]{hj13}), there is some $\varepsilon_0>0$ such that all matrices in $B(M,\varepsilon_0)$ have $n$ different eigenvalues, and this implies that $B(M,\varepsilon_0)\subseteq \bigcup_{t=0}^{\lfloor n/2\rfloor}\bun_t$. Now, assume, by contradiction, that, for all $\varepsilon<\varepsilon_0$, the open ball $B(M,\varepsilon)$ is not contained in $\bun_{t_0}$, which implies, since all matrices in $B(M,\varepsilon)$ have $n$ different eigenvalues, that $B(M,\varepsilon)\cap \left(\bigcup_{t\neq t_0}\bun_t\right)\neq\emptyset$. As a consequence, $M$ belongs to the closure of $\bigcup_{t\neq t_0}\bun_t$, which is equal to $\bigcup_{t\neq t_0}\overline{\bun_t}$. Hence, $M\in\bun_{t_0}\cap\overline{\bun_t}$, for some $t\neq t_0$, which is in contradiction with part (ii).

Finally, let us prove claim (i) in the statement. For this, we strongly rely on the developments in \cite[Ch. VIII]{gant}. More precisely, the dimension of the solution space of $XA-AX=0$ depends on RJCF($A$) and, for RJCF$(A)=\bigoplus_{i=1}^tC(a_i,b_i)\oplus\bigoplus_{i=1}^{n-2t}[c_i]$, with the parameters $a_i, b_i,$ and $ c_i$ as in the statement, it is equal to
$$
\sum_{i=1}^t \dim_\R\{ X\in\R^{2\times 2} \, : \, XC(a_i,b_i)-C(a_i,b_i)X=0 \}+\sum_{i=1}^{n-2t}\dim_\R\{ x\in\R \, : \, xc_i-c_ix=0 \}.
$$
Since $\dim_\R\{ x\in\R \, : \, xc_i-c_ix=0 \} =1$ and, by a straightforward calculation, $\dim_\R\{ X\in\R^{2\times 2} \, : \, XC(a_i,b_i)-C(a_i,b_i)X=0 \}=2$, we conclude from \eqref{eq.codimorb} that
$$
\cod\orbr\left(\bigoplus_{i=1}^tC(a_i,b_i)\oplus\bigoplus_{i=1}^{n-2t}[c_i]\right)=2t+ n - 2t=n,
$$
and, as a consequence of \eqref{eq.codimbund}, $\cod\bunr\left(\bigoplus_{i=1}^tC(a_i,b_i)\oplus\bigoplus_{i=1}^{n-2t}[c_i]\right)=n-n=0$.
\end{proof}

\section{ Real eigenvalues of random matrices: Numerical experiments}\label{random.sec}

In this last section we provide some numerical experiments to support our main result, i.e., Theorem \ref{main.th}, and to connect it with some results available in the literature on real eigenvalues of random matrices. Let us refer to a real $n\times n$ matrix whose entries are i.i.d. random variables with standard normal distribution as a ``random matrix". The probability that a random matrix has exactly $k$ real eigenvalues has been obtained in \cite{edelman97}. Following the notation in \cite{edelman97}, we denote this probability by $p_{n,k}$. According to Theorem \ref{main.th}, whenever $k$ and $n$ have the same parity, it must be $p_{n,k}\neq0$. The values of $p_{n,k}$ for $n=8$ and $n=9$ are provided in Table \ref{pnk.table}.

\begin{table}[]
    \centering
    \begin{tabular}{|c|c|c|}\hline
    $n$&$k$&$p_{n,k}$\\\hline\hline
         \multirow{5}{1em}{$8$} & $8$ & $6.10\cdot10^{-5}$ \\
& $6$ & $2.05\cdot10^{-2}$ \\
& $4$ & $3.46\cdot10^{-1}$ \\&$2$ &$5.71\cdot10^{-1}$  \\
         &$0$& $6.21\cdot10^{-2}$\\\hline
         \multirow{5}{1em}{$9$} & $9$ & $3.81\cdot10^{-6}$ \\
& $7$ & $2.56\cdot10^{-3}$ \\
& $5$ & $1.46\cdot10^{-1}$ \\&$3$ &$5.93\cdot10^{-1}$  \\
         &$1$& $2.57\cdot10^{-1}$\\\hline
    \end{tabular}
    \caption{The value of $p_{n,k}$ for $n=8$ and $n=9$ (from \cite[Table 1]{edelman97}).}
    \label{pnk.table}
\end{table}

We have computed the number of real eigenvalues of real random matrices using the following {\sc matlab} code:

\begin{sol}
\begin{verbatim}
function realevals(n,m)
% counts the number of real evals of m random nxn matrices
counter=zeros(m,1);
for i=1:m
    a=randn(n);
    e=eig(a);
    normi=abs(e);
    for j=1:n
        if abs(e(j)/normi(j)-1)<=eps*cond(a)
            counter(i)=counter(i)+1;
        elseif abs(e(j)/normi(j)+1)<=eps*cond(a)
            counter(i)=counter(i)+1;
        else
            counter(i)=counter(i);
        end
    end
end

x = unique(counter);
N = numel(x);
count = zeros(N,1);
for k = 1:N
   count(k) = sum(counter==x(k));
end
disp([ x(:) count ]);
\end{verbatim}
\end{sol}

The results, for $10^7$ test matrices for each size $n=8, 9, 10, 15$, which have been computed with {\sc matlab} R2024b, are displayed in Table \ref{results.table}. The column $k$ is the number of real eigenvalues, and the column $F$ denotes the frequency. The last column displays the ratio between the frequency and the total number of tests, which is an experimental approximation to the probability $p_{n,k}$. For $n=8$ and $n=9$ the experimental results are very close to the corresponding values $p_{n,k}$ in Table \ref{pnk.table}. Actually, for some values of $k$ they sharply coincide with the theoretical ones up to three digits of accuracy. For $n=10$ all possible numbers of real eigenvalues occur (namely $k=0,2,4,6,8,$ and $10)$. Therefore, for $n=8,9,$ and $10$, all generic bundles described in Theorem \ref{main.th} show up. However, for $n=15$ only up to $k=11$ real eigenvalues appear (so $k=13$ and $k=15$ are missing). In both cases, the results are in accordance with the expected number of real eigenvalues obtained in \cite{eks94}. More precisely, the expectation for $n=10$ is, approximately, $2.93$ (see Table 1 in \cite{eks94}), whose closest even number is $k=2$, which is the one with the highest frequency, whereas for $n=15$ it can be calculated from the formula in \cite[Cor. 5.3]{eks94} and gives, approximately, $3.51$, whose closest odd number is $k=3$, which is, again, the one with largest frequency.

\begin{table}[]
    \centering
    \begin{tabular}{|c|c|c|c|}\hline
    $n$&$k$&$F$&$F/10^7$\\\hline\hline
          \multirow{5}{1em}{$8$} & $8$ & $594$&$5.94\cdot10^{-5}$ \\
& $6$ & $205759$&$2.06\cdot10^{-2}$ \\
& $4$ & $3456910$&$3.46\cdot10^{-1}$  \\&$2$ &$5713924$& $5.71\cdot10^{-1}$\\
         &$0$& $622813$&$6.23\cdot10^{-2}$\\\hline
         \multirow{5}{1em}{$9$} & $9$ & $46$&$4.6\cdot10^{-6}$ \\
& $7$ & $35384$&$3.53\cdot10^{-3}$ \\
& $5$ & $1462469$&$1.46\cdot10^{-1}$ \\&$3$ &$5931622$&$5.93\cdot10^{-1}$  \\
         &$1$& $2570479$&$2.57\cdot10^{-1}$\\\hline
          \multirow{6}{1em}{$10$} &$10$&$2$&$2\cdot10^{-7}$\\
          & $8$ & $4325$&$4.22\cdot10^{-4}$ \\
& $6$ & $444855$&$4.45\cdot10^{-2}$ \\
& $4$ & $4172775$&$4.17\cdot10^{-1}$  \\&$2$ &$4944333$& $4.94\cdot10^{-1}$\\
         &$0$& $433710$&$4.34\cdot10^{-2}$\\\hline
         \multirow{6}{1em}{$15$} &$11$&$12$&$1.2\cdot10^{-6}$\\& $9$ & $5444$&$5.44\cdot10^{-4}$ \\
& $7$ & $335896$&$3.36\cdot10^{-2}$ \\
& $5$ & $3142390$&$3.14\cdot10^{-1}$ \\&$3$ &$5248744$&$5.25\cdot10^{-1}$  \\
         &$1$& $1267514$&$1.27\cdot10^{-1}$\\\hline
    \end{tabular}
    \caption{Number of real eigenvalues of $10^7$ real random matrices with size $n\times n$.}
    \label{results.table}
\end{table}

We can slightly force the random matrices for $n=15$ in order to get positive frequencies for $k=13$ and $k=15$, namely for $13$ and $15$ real eigenvalues to show up. For this, we add to each random matrix a diagonal matrix with the $(i,i)$ entry being equal to $2i$. The results for these matrices are displayed in Table \ref{trick.table}. As it can be seen, in this case all possible numbers of real eigenvalues (namely, $k$ real eigenvalues, with $k$ being any odd number from $1$ to $15$) show up, which confirms the genericity of the bundles described in Theorem \ref{main.th} for $n=15$.

\begin{table}[]
    \centering
    \begin{tabular}{|c|c|c|c|}\hline
    $k$&$F$&$F/10^6$\\\hline\hline
    $15$&$11240$&$1.12\cdot10^{-2}$\\
    $13$&$91675$&$9.17\cdot10^{-2}$\\
         $11$&$269714$&$2.70\cdot10^{-1}$\\ $9$ & $352263$&$3.52\cdot10^{-1}$ \\
 $7$ & $212891$&$2.13\cdot10^{-1}$ \\
 $5$ & $56619$&$5.66\cdot10^{-2}$ \\$3$ &$5477$&$5.48\cdot10^{-3}$  \\
         $1$& $121$&$1.21\cdot10^{-4}$\\\hline
    \end{tabular}
    \caption{Number of real eigenvalues of $10^6$ real matrices of the form \texttt{randn(15)+diag(2,4,6,...,2*15)}.}
    \label{trick.table}
\end{table}

\bigskip

\noindent{\bf Acknowledgments}. This work has been partially supported by grants PID2023-147366NB-I00 funded by MICIU/AEI/10.13039/501100011033 and FEDER/UE, and RED2022-134176-T.

\bibliographystyle{plain}

\end{document}